\documentclass[a4paper,11pt,twoside]{amsart}
\usepackage[utf8]{inputenc}
\usepackage{amssymb,amsmath,amsthm,amscd}
\usepackage{mathrsfs}
\usepackage{bbm}
\usepackage{enumerate}
\usepackage{fullpage}
\usepackage{comment}

\usepackage{lmodern}

\newcommand{\C}{\ensuremath{\mathbb{C}}}
\newcommand{\N}{\ensuremath{\mathbb{N}}}
\newcommand{\cC}{\mathcal{C}}
\newcommand{\cH}{\mathcal{H}}
\newcommand{\cK}{\mathcal{K}}
\newcommand{\cO}{\mathcal{O}}

\newcommand{\id}{\mathord{\text{\rm id}}}
\newcommand{\G}{\mathbb{G}}
\newcommand{\al}{\alpha}
\newcommand{\ot}{\otimes}
\newcommand{\eps}{\varepsilon}
\newcommand{\cB}{\mathcal{B}}

\renewcommand{\ker}{\operatorname{Ker}}
\newcommand{\ran}{\operatorname{Ran}}
\newcommand{\Ghat}{\hat{\mathbb G}}
\newcommand{\vNaG}{L^{\infty}(\mathbb G)}
\newcommand{\uniG}{C_u(\mathbb G)}
\newcommand{\uniC}{C_u(\mathcal C)}

\newcommand{\1}{\mathbbm{1}}
\newcommand{\tube}{\mathcal{AC}}
\newcommand{\cA}{\mathcal{A}}
\newcommand{\be}{\beta}
\newcommand{\ga}{\gamma}
\newcommand{\ep}{\varepsilon}

\DeclareMathOperator{\Rep}{Rep}
\DeclareMathOperator{\Irr}{Irr}
\DeclareMathOperator{\Mor}{Mor}
\DeclareMathOperator{\End}{End}

\DeclareMathOperator{\mult}{mult}

\DeclareMathOperator{\ind}{ind}

\theoremstyle{definition}
\newtheorem{theorem}{Theorem}[section]

\newtheorem{definition}[theorem]{Definition}
\newtheorem{lemma}[theorem]{Lemma}
\newtheorem{proposition}[theorem]{Proposition}
\newtheorem{corollary}[theorem]{Corollary}
\newtheorem{remark}[theorem]{Remark}

\author{Yuki Arano}
\thanks{YA was supported by a fellowship of the Japan Society for the Promotion of Science and the Program for Leading Graduate Schools, MEXT, Japan}
\address{Yuki Arano
\newline Department of Mathematics, Kyoto university
\newline Kitashirakawa-Oiwakecho, Sakyo-ku, Kyoto, 606-8502, Japan}
\email{y.arano@math.kyoto-u.ac.jp}

\author{Tim de Laat}
\thanks{TdL is supported by the Deutsche Forschungsgemeinschaft (SFB 878)}
\address{Tim de Laat
\newline Mathematisches Institut, Westf\"alische Wilhelms-Universit\"at M\"unster
\newline Einsteinstrasse 62, 48149, Germany}
\email{tim.delaat@uni-muenster.de}

\author{Jonas Wahl}
\thanks{JW is supported by European Research Council Consolidator Grant 614195}
\address{Jonas Wahl
\newline KU Leuven, Department of Mathematics
\newline Celestijnenlaan 200B -- Box 2400, 3001 Leuven, Belgium}
\email{jonas.wahl@kuleuven.be}

\title{The Fourier algebra of a rigid $C^{\ast}$-tensor category}

\begin{document}

\begin{abstract}
  Completely positive and completely bounded mutlipliers on rigid $C^{\ast}$-tensor categories were introduced by Popa and Vaes. Using these notions, we define and study the Fourier-Stieltjes algebra, the Fourier algebra and the algebra of completely bounded multipliers of a rigid $C^{\ast}$-tensor category. The rich structure that these algebras have in the setting of locally compact groups is still present in the setting of rigid $C^{\ast}$-tensor categories. We also prove that Leptin's characterization of amenability still holds in this setting, and we collect some natural observations on property (T).
\end{abstract}

\maketitle

\section{Introduction}
Let $G$ be a locally compact group. The Fourier algebra $A(G)$ of $G$ is the Banach algebra consisting of the matrix coefficients of the left-regular representation of $G$, and the Fourier-Stieltjes algebra $B(G)$ of $G$ is the Banach algebra consisting of the matrix coefficients of all unitary representations of the group. Both these algebras were introduced by Eymard \cite{eymard}. Nowadays, they play an important role in analytic group theory, in particular in the study of approximation and rigidity properties for groups (see e.g.~\cite{brownozawa}). In this respect, also their relation to operator algebras is fundamental: the dual of $A(G)$ can be identified with the group von Neumann algebra $L(G)$ of $G$, and $B(G)$ can be identified with the dual of the universal group $C^{\ast}$-algebra $C^{\ast}(G)$ of $G$. In the study of approximation and rigidity properties for groups, also the Banach algebra of completely bounded Fourier multipliers plays a fundamental role. Analogues of the Fourier-Stieltjes algebra, the Fourier algebra and the algebra of completely bounded Fourier multipliers have been studied quite extensively in the setting of quantum groups, starting with the work of Daws \cite{daws}.

A couple of years ago, fundamentally new results on approximation and rigidity properties for quantum groups were proven \cite{decommerfreslonyamashita}, \cite{arano1}. Partly relying on these results, Popa and Vaes formulated the theory of unitary representations for ``subfactor related group-like objects'' (e.g.~quantum groups, subfactors and $\lambda$-lattices), in the setting of rigid $C^{\ast}$-tensor categories \cite{popavaes} (see also \cite{neshveyevyamashita} and \cite{ghoshjones}). The representation category of a compact quantum group and the standard invariant of a subfactor are important and motivating examples of rigid $C^{\ast}$-tensor categories. Intimately related to the unitary representation theory are the notions of completely positive and completely bounded multipliers, which are particularly important in the study of approximation and rigidity properties. Popa and Vaes studied such properties for subfactor related group-like objects in \cite{popavaes}. Recently, more new results on approximation and rigidity properties for subfactor related group-like objects were proven \cite{cjones}, \cite{aranovaes}, \cite{aranodelaatwahl}, \cite{tarragowahl}.

The aim of this article is to study the structure of the collections of completely positive and completely bounded multipliers. In particular, we define and study analogues of the Fourier-Stieltjes algebra, the Fourier algebra and the algebra of completely bounded Fourier multipliers in the setting of rigid $C^{\ast}$-tensor categories. It turns out that in this setting they also form Banach algebras, and that the operator algebraic structure of these algebras is still present.

It turns out that in the setting of rigid $C^{\ast}$-tensor categories, we still have Leptin's characterization of amenability (see Theorem \ref{thm:leptin}). Moreover, we collect some observations on property (T) for quantum groups and rigid $C^{\ast}$-tensor categories (see Section \ref{sec:propertyt}).

\section*{Acknowledgements}
The authors thank Stefaan Vaes for valuable discussions, suggestions and remarks, and for permitting them to include his proof of Theorem \ref{thm:fima}.

\section{Preliminaries}
\subsection{Rigid $C^{\ast}$-tensor categories} \label{subsec:cstc}
A $C^{\ast}$-tensor category is a category that behaves similar to the category of Hilbert spaces. For the basic theory of $C^{\ast}$-tensor categories and the facts mentioned in this subsection, we refer to \cite[Chapter 2]{neshveyevtuset}.

In what follows, all tensor categories will be assumed to be strict, unless explicitly mentioned otherwise. This is not a fundamental restriction, since every tensor category can be strictified.

Let $\mathcal{C}$ be a $C^{\ast}$-tensor category. An object $\bar{u}$ in $\mathcal{C}$ is conjugate to an object $u$ in $\mathcal{C}$ if there are $R \in \Mor(\1,  \bar{u} \ot u)$ and $\bar{R} \in \Mor(\1, u \ot \bar{u})$ such that
\[
  u \xrightarrow{1 \ot R} u \ot \bar{u} \ot u \xrightarrow{\bar{R}^* \ot 1 } u \ \ \text{and} \ \ \bar{u} \xrightarrow{1 \ot \bar{R}} \bar{u} \ot u \ot \bar{u} \xrightarrow{R^* \ot 1} \bar{u}
\]
are the identity morphisms. Conjugate objects are uniquely determined up to isomorphism. If every object has a conjugate object, then the category $\mathcal{C}$ is called a rigid $C^{\ast}$-tensor category.

Let $\mathrm{Irr}(\mathcal{C})$ denote the set of equivalence classes of irreducible objects in $\mathcal{C}$. Using the same notation as above, if $u$ is an irreducible object with a conjugate, then $d(u) = \|R \| \| \bar{R} \|$ is independent of the choice of the morphisms $R$ and $\bar{R}$. An arbitrary object $u$ in a rigid $C^{\ast}$-tensor category is unitarily equivalent to a direct sum $u \cong \bigoplus_{k} u_k$ of irreducible objects, and we put $d(u) = \sum_{k} d(u_k)$. The function $d : \mathcal{C} \to [0, \infty)$ defined in this way is called the intrinsic dimension of $\mathcal{C}$.   

\subsection{Multipliers on rigid $C^{\ast}$-tensor categories}
Multipliers on rigid $C^{\ast}$-tensor categories were introduced by Popa and Vaes \cite{popavaes}.
\begin{definition} \label{mult}
A multiplier on a rigid $C^{\ast}$-tensor category $\cC$ is a family of linear maps 
\[
  \theta_{\alpha,\beta} : \End(\alpha \ot \beta) \to  \End(\alpha \ot \beta)
\]
indexed by $\alpha, \beta \in \cC$ such that
\begin{align}
  \theta_{\alpha_2,\beta_2}(UXV^{\ast}) &= U\theta_{\alpha_1,\beta_1}(X)V^{\ast}, \nonumber \\
  \theta_{\alpha_1 \ot \alpha_2,\beta_1 \ot \beta_2} (1 \ot X \ot 1) &= 1 \ot \theta_{\alpha_2, \beta_1}(X) \ot 1 \label{eq:equation1}
\end{align}
for all $\alpha_i, \beta_i \in \cC, X \in \End(\alpha_2 \ot \beta_1)$ and $U,V \in \mathrm{Mor}(\alpha_1,\alpha_2) \ot \mathrm{Mor}(\beta_2,\beta_1)$.
\end{definition}
A multiplier $(\theta_{\alpha,\beta})$ is said to be completely positive (or a cp-multiplier) if all maps $\theta_{\alpha,\beta}$ are completely positive. A multiplier $(\theta_{\alpha,\beta})$ is said to be completely bounded (or a cb-multiplier) if all maps $\theta_{\alpha,\beta}$ are completely bounded and $\| \theta \|_{\mathrm{cb}} = \sup_{\alpha, \beta \in \cC} \| \theta_{\alpha,\beta}\|_{\mathrm{cb}} < \infty$. By \cite[Proposition 3.6]{popavaes}, every multiplier $(\theta_{\alpha, \beta})$ is uniquely determined by a family of linear maps $\Mor(\alpha \ot \bar{\alpha},\1) \to \Mor(\alpha \ot \bar{\alpha}, \1), \ \alpha \in \Irr(\cC)$. Since $\Mor(\alpha \ot \bar{\alpha}, \1)$ is one-dimensional whenever $\alpha$ is irreducible, each of these linear maps is given by multiplication with a scalar $\varphi(\alpha) \in \C, \ \alpha \in \Irr(\cC),$ and hence every multiplier corresponds uniquely to a function $\varphi: \Irr(\cC) \to \C$. Therefore, when we speak of a multiplier we will often mean the underlying function $\varphi: \Irr(\cC) \to \C$.

\subsection{The fusion algebra and admissible $\ast$-representations} \label{subsec:admissiblerepresentations}
Recall that the fusion algebra $\C[\cC]$ of a rigid $C^{\ast}$-tensor category $\cC$ is defined as the free vector space with basis $\Irr(\cC)$ and multiplication given by
\[
  \alpha \beta = \sum_{\gamma \in \Irr(\cC)} \mult(\alpha \ot \beta, \gamma) \gamma, \ \ \ \alpha, \beta \in \Irr(\cC).
\]
In fact, the fusion algebra is a $\ast$-algebra when equipped with the involution $\alpha^{\sharp}=\bar{\alpha}$.

In \cite{popavaes}, Popa and Vaes defined the notion of admissible $\ast$-representation of $\mathbb{C}[\mathcal{C}]$ as a unital $\ast$-representation $\Theta: \C[\cC] \to B(\cH)$ such that for all $\xi \in \cH$ the map
\[
  \Irr(\cC) \to \C, \; \alpha \to d(\alpha)^{-1} \langle \Theta(\alpha) \xi, \xi \rangle
\]
is a cp-multiplier. Moreover, they proved the existence of a universal admissible $\ast$-representation and denoted the corresponding enveloping $C^{\ast}$-algebra of $\C[\cC]$ by $\uniC$.

\subsection{The tube algebra}
In \cite{ghoshjones}, Ghosh and Jones related the representation theory of rigid $C^{\ast}$-tensor categories to Ocneanu's tube algebra, which was introduced in \cite{ocneanu}. More precisely, Ghosh and Jones proved that a representation of $\C[\cC]$ is admissible in the sense of Popa and Vaes if and only if it is unitarily equivalent to the restriction of a $*$-representation of the tube algebra to $\C[\cC]$. While we do not elaborate on this in detail, the tube algebra picture is convenient when studying completely bounded multipliers (see in particular Proposition \ref{ArVaes}).

Let us recall the definition of the tube algebra. Let $\cC$ be a rigid $C^{\ast}$-tensor category. For each equivalence class $\alpha \in \mathrm{Irr}(\mathcal{C})$, choose a representative $X_{\alpha} \in \alpha$, and let $X_0$ denote the representative of the tensor unit. Moreover, let $\Lambda$ be a countable family of equivalence classes of objects in $\cC$ with distinct representatives $Y_{\beta} \in \beta$ for every $\beta \in \Lambda$. The annular algebra with weight set $\Lambda$ is defined as
\[
  \cA \Lambda = \bigoplus_{\al, \be \in \Lambda, \ \ga \in \Irr(\cC)} \Mor(X_{\ga} \ot Y_{\al}, Y_{\be} \ot X_{\ga}).
\]
The algebra $\cA \Lambda$ comes equipped with the structure of an associative $*$-algebra. We will always assume the weight set $\Lambda$ to be full, i.e.~every irreducible object is equivalent to a subobject of some element in $\Lambda$. The annular algebra with weight set $\Lambda = \Irr(\cC)$ is called the tube algebra of Ocneanu, and we write $\cA \Lambda = \tube$.

\subsection{Unitary half braidings} \label{sec:unitaryhalfb}
Another approach to the representation theory of a rigid $C^*$-tensor category was developed in \cite{neshveyevyamashita} in terms of unitary half braidings on ind-objects. This approach is particularly well behaved when one is interested in taking tensor products of representations, a fact we will make use of in the proof of Theorem \ref{FourierStieltjesBanach}. Let us recall that intuitively, an ind-object $X \in \ind \cC$ is a possibly infinite direct sum of objects in the rigid $C^*$-tensor category $\cC$ and that $\ind \cC$ is a $C^*$-tensor category containing $\cC$, albeit generically not a rigid one. For a rigorous definition and additional details, see \cite{neshveyevyamashita}. A unitary half braiding $\sigma$ on an ind-object $X \in \ind \cC$ was defined in \cite{neshveyevyamashita} as a family of unitary morphisms $\sigma_{\alpha} \in \Mor ( \alpha \ot X, X \ot \alpha ), \ \alpha \in  \cC $ satisfying
\begin{itemize}
\item $\sigma_{\1} = \id$;
\item $(1 \ot V) \sigma_{\alpha} = \sigma_{\beta} (V \ot 1)$ for all $V \in \Mor(\alpha,\beta)$;
\item $\sigma_{\alpha \ot \beta} = (\sigma_{\alpha} \ot 1) (1 \ot \sigma_{\beta})$.
\end{itemize}
Every pair $(X, \sigma)$ consisting of an ind-object $X$ and a unitary half braiding $\sigma$ on $X$, defines a $*$-representation of $\C[\cC]$ on the Hilbert space $\cH_{(X, \sigma)} = \Mor_{\ind \cC}(\1, X)$ with inner product $\langle \xi, \eta \rangle 1 = \eta^* \xi$. More concretely, if we choose a set of representatives $Y_{\alpha}$ for $\al \in \Irr(\cC)$ with standard solution of the conjugate equations $(R_{Y_{\alpha}}, \bar{R}_{Y_{\alpha}} )$, then 
\[ \pi_{(X, \sigma)}: \C[\cC] \to B(\cH_{(X, \sigma)}), \quad \pi(\alpha) \xi = (1 \ot \bar{R}_{Y_{\alpha}}^*)(\sigma_{Y_{\alpha}} \ot 1)(1 \ot \xi \ot 1) \bar{R}_{Y_{\alpha}} \]
defines a $*$-representation. Note that a different choice of representatives yields a unitarily equivalent $*$-representation. It was shown in \cite{neshveyevyamashita} that any admissible representation is unitarily equivalent to a representation of the above form. More generally, for an explicit bijection between unitary half braidings on ind-objects and (non-degenerate) $*$-representations of the tube algebra, see \cite[Proposition 3.14]{popashlyakhtenkovaes}.

\section{The Fourier-Stieltjes algebra} \label{sec:fourieralgebra}
Let $\mathcal{C}$ be a rigid $C^{\ast}$-tensor category, and let $\mathbb{C}[\mathcal{C}]$ denote its fusion algebra. The notion of admissible $\ast$-representation and the universal admissible $\ast$-representation, as introduced by Popa and Vaes in \cite{popavaes}, were recalled in Section \ref{subsec:admissiblerepresentations}. Admissible $\ast$-representations can be used to define the Fourier-Stieltjes algebra of a $C^{\ast}$-tensor category.
\begin{definition}
The Fourier-Stieljes algebra $B(\cC)$ of a rigid $C^{\ast}$-tensor category $\cC$ is the algebra of functions $\varphi: \Irr(\mathcal{C}) \to \C$ of the form
\[
  \varphi(\alpha) = d(\alpha)^{-1} \langle \Theta(\alpha) \xi, \eta \rangle \ \ \ (\alpha \in \Irr(\mathcal{C})),
\]
where $\Theta : \mathbb{C}[\cC] \to B(\cK)$ is an admissible $\ast$-representation of the fusion algebra and $\xi, \eta \in \cK$. We call such a function $\varphi$ a (matrix) coefficient of $\Theta$. The algebra structure is given by pointwise multiplication.
\end{definition}
\begin{remark}
To see that the Fourier-Stieltjes algebra is a unital algebra, note that it can also be defined as the span of the cp-multipliers $CP(\mathcal{C})$ on $\cC$, i.e.
\[
  \cB(\cC) = \left\{ \sum_{i=1}^n \lambda_i \varphi_i \;\Bigg\vert\; n \in \N, \  \lambda_i \in \mathbb{C}, \ \varphi_i \in CP(\cC), \ i=1,\dots,n \right\}.
\]
Indeed, it follows from the definition of cp-multiplier that the product of two such multipliers is a cp-multiplier again.
\end{remark}
We will now equip $\cB(\cC)$ with a norm that turns it into a Banach algebra.
\begin{proposition}\label{FSa}
The map $\Phi_0 : CP(\cC) \to \uniC^{\ast}_+, \ \Phi_0(\varphi)(\alpha) = \omega_{\varphi}(\alpha) = d(\alpha) \varphi(\alpha)$ extends linearly to an isomorphism of vector spaces $\Phi: \cB(\cC) \to \uniC^{\ast} $. Moreover, for an element $\varphi \in \cB(\cC)$, we have the following equality of norms:
\[
  \| \varphi \|_{\cB(\cC)} := \| \Phi(\varphi) \| = \min \{\| \xi \| \| \eta \| \mid \ \varphi(\cdot) = d(\cdot)^{-1} \langle \Theta(\cdot) \xi, \eta \rangle, \ \Theta \text{ admissible} \}.
\]
\end{proposition}
\begin{proof}
By definition of $\uniC$ and \cite[Proposition 4.2]{popavaes}, the map $\Phi_0$ is well-defined, and so is $\Phi$. It is clear that $\Phi$ defines a bijection. The second part follows directly from the following lemma.
\end{proof}

\begin{lemma}
Let $A$ be a unital $C^*$-algebra. For all $\omega \in A^*$, we have the following equality of norms:
\[ \| \omega \| = \min \{\| \xi \| \| \eta \| \mid \ \omega(\cdot) = \langle \Theta(\cdot) \xi, \eta \rangle, \ \Theta \ *-\text{representation of } A \}. \]
\end{lemma}
Although this is a well-known result, for the sake of completeness, we include a proof.
\begin{proof}
 Since we can view $A^{\ast}$ as the predual of the von Neumann algebra $A^{**}$, we can consider the polar decomposition of $\omega$, i.e.~there exists a positive normal functional $| \omega | \in A^{\ast}_+ $ and a partial isometry $V \in A^{**}$ such that $\omega = V |\omega|$ and $\| \omega \| = \| \ | \omega | \ \| $. Consider the GNS-representation $\Theta: A \to B(\cK)$ of $|\omega|$, which has a cyclic vector, say $\eta$, i.e.~$|\omega|(x) = \langle \Theta(x) \eta, \eta \rangle$ for all $x \in A$. As a consequence, we obtain that
\[
  \omega(x) = (V |\omega|)(x) = |\omega|(xV) = \langle \Theta(x) \Theta'(V) \eta, \eta \rangle \ \ \forall x \in A,
\]
where $\Theta'$ is the unique extension of $\Theta$ to $A^{**}$. Defining $\xi = \Theta'(V) \eta $, we have $\| \xi \| \leq \| \eta \|$, since $V$ is a partial isometry. Altogether, the set on the right side of the equality which is to be proven is nonempty and we have $\| \omega \| =  \| \ | \omega | \ \| = \| \eta \|^2 \geq \| \xi \| \| \eta \|$. \\
On the other hand, for every $*$-representation $\Theta:A \to B(\cH)$ and $\xi, \eta \in \cH$ such that $\omega(\cdot) = \langle \Theta(\cdot) \xi, \eta \rangle$ we have
\[ | \omega(x) | = | \langle \Theta(x) \xi, \eta \rangle | \leq \| x \| \| \xi \| \| \eta \| \quad \forall x \in A.  \]
\end{proof}

\begin{theorem} \label{FourierStieltjesBanach}
Let $\cC$ be a rigid $C^{\ast}$-tensor category. Then $\cB(\cC)$ is a Banach algebra with respect to the norm defined in the previous proposition.
\end{theorem}
\begin{proof}
The definition of $\| \cdot \|_{\cB(\cC)}$ directly implies that $(\cB(\cC), \| \cdot \|_{\cB(\cC)})$ is a Banach space and hence we are only left with showing that $\| \varphi_1 \varphi_2  \|_{\cB(\cC)} \leq \| \varphi_1 \|_{\cB(\cC)} \| \varphi_2 \|_{\cB(\cC)}$ for $\varphi_1, \varphi_2 \in \cB(\cC)$. Now, by Proposition \ref{FSa} and the discussion in Section \ref{sec:unitaryhalfb}, for $i=1,2$ we can find pairs $(X_i, \sigma_i)$ of ind-objects $X_i \in \ind \cC$ and unitary half braidings $\sigma_i$ on $X_i$ as well as $\xi_i, \eta_i \in \cH_{(X_i,\sigma_i)}$ such that
\[ \varphi_i(\alpha) = d(\alpha)^{-1} \langle \pi_{(X_i,\sigma_i)}(\alpha) \xi_i, \eta_i \rangle \quad \text{and} \quad \| \varphi_i \|_{\cB(\cC)} = \| \xi_i \| \| \eta_i \|. \]
Following \cite{neshveyevyamashita}, $\sigma = (1 \ot \sigma_2)(\sigma_1 \ot 1)$ defines a unitary half braiding on $X = X_1 \ot X_2 \in \ind \cC$. Recall from \cite{neshveyevyamashita} that, in the same way as unitary half braidings are generalizations of group representations, this new half braiding is the proper analogue of the tensor product of the unitary half braidings $(X_1, \sigma_1)$ and $(X_2, \sigma_2)$. We have $\xi = (\xi_1 \ot 1) \xi_2, \eta = (\eta_1 \ot 1) \eta_2 \in \Mor_{\ind \cC}(\1, X_1 \ot X_2) = \cH_{(X,\sigma)}$ with $\| \xi \| = \| \xi_1 \| \| \xi_2 \|, \| \eta \| = \| \eta_1 \| \| \eta_2 \|. $ Choosing representatives $Y_{\alpha}$ for irreducible objects $\al \in \Irr(\cC)$ as in Section \ref{sec:unitaryhalfb} and using the fact that $\bar{R}_{Y_{\alpha}}^* \bar{R}_{Y_{\alpha}} = d(\alpha) \in \Mor(\1, \1)$, we compute 
\[
  \pi_{(X,\sigma)}(\alpha) \xi = d(\alpha)^{-1} (\pi_{(X_1,\sigma_1)}(\alpha) \xi_1 \ot 1) (\pi_{(X_2,\sigma_2)}(\alpha) \xi_2),
\]
and hence
\[
  \varphi_1(\alpha) \varphi_2(\alpha) = d(\alpha)^{-1} \langle \pi_{(X,\sigma)}(\alpha) \xi, \eta \rangle,
\]
which finishes the proof.
\end{proof}

\section{The Fourier algebra}
Recall that by \cite[Corollary 4.4]{popavaes}, the left regular representation of $\mathbb{C}[\cC]$ given by
\[
  \lambda: \mathbb{C}[\cC] \to B(\ell^2(\Irr(\cC))), \ \lambda(\alpha) \delta_{\beta} = \sum_{\gamma \in \Irr(\cC)} \mult(\alpha \ot \beta, \gamma) \delta_{\gamma}
\]
is admissible and corresponds to the cp-multiplier defined by $\varphi_{\lambda}(\alpha) = \delta_{\alpha, \1} \ (\alpha \in \Irr(\cC))$.
\begin{definition}
The Fourier algebra $A(\cC)$ of a rigid $C^{\ast}$-tensor category $\cC$ is defined as the predual of the von Neumann algebra $\lambda(\mathbb{C}[\cC])''$.
\end{definition}
Recall that there is a one-to-one correspondence between functions on $\Irr(\cC)$ and functionals $\omega : \mathbb{C}[\cC] \to \mathbb{C}$ given by $\varphi \mapsto \omega_{\varphi}$, where $\omega_{\varphi}(\al) = d(\al) \varphi(\al)$. By this correspondence, $A(\cC)$ can also be interpreted as an algebra of functions on $\Irr(\cC)$.
\begin{proposition} \label{Fa}
For every $\omega \in A(\cC)$, there exist $\xi, \eta \in \ell^2(\Irr(\cC))$ such that $\omega(x) = \langle \lambda(x) \xi, \eta \rangle$. In addition, 
\[ \| \omega \|_{A(\cC)} = \min \{\| \xi \| \| \eta \| \mid \omega(\cdot) = \langle \lambda(\cdot) \xi, \eta \rangle, \ \xi,\eta \in \ell^2(\Irr(\cC)) \}. \]
\end{proposition}
\begin{proof}
Since $M = \lambda(\mathbb{C}[\cC])''$ is nothing but the GNS-representation with respect to $\omega_{\varphi}$, where $\varphi_{\lambda}(\alpha) = \delta_{\alpha, \1} \ (\alpha \in \Irr(\cC))$, we can represent every positive normal functional on $M$ as a vector state on $M$ by \cite[Chapter IX, Lemma 1.6]{takesaki}. The result for a general normal functional follows as in Proposition \ref{FSa} by polar decomposition.
\end{proof}
\begin{remark}
It is an immediate consequence of Proposition \ref{Fa} that we have
\[ \| \varphi \|_{B(\cC)} \leq \| \varphi \|_{A(\cC)} \]
 for $ \varphi \in A(\cC)$, and it is not hard to see that the norms are actually equal. Indeed, the dual $C_r(\cC)^*$ of the reduced $C^*$-algebra $C_r(\cC)=\overline{\lambda(\mathbb{C}[\mathcal{C}])}$ identifies isometrically with the dual of a quotient of $C_u(\mathcal{C})$ and hence with the annihilator of a closed ideal in $\uniC$. Consequently, 
 \[ \| \varphi \|_{A(\cC)} = \| \varphi \|_{C_r(\cC)^*} =\| \varphi \|_{\uniC^*} \]
 for $ \varphi \in A(\cC)$.
This means that we could also have defined $A(\cC)$ as the closure of the coefficients of the left regular representation in $\cB(\cC)$. Moreover, we will see in Corollary \ref{CorCB} that $A(\cC)$ is a closed ideal in $\cB(\cC)$ and in particular a Banach algebra itself.
\end{remark}

\section{Completely bounded multipliers} \label{sec:cbm}
In this section, we study the algebra of completely bounded multipliers 
\[
  M_0A(\cC) = \{ \varphi : \Irr(\cC) \to \mathbb{C} \mid \varphi \ \text{cb-multiplier} \}.
\]

While the Fourier algebra $A(\cC)$ is only defined in terms of the fusion algebra $\mathbb{C}[\cC]$, the Fourier-Stieltjes algebra $B(\mathcal{C})$ and the algebra $M_0A(\mathcal{C})$ of completely bounded multipliers use considerably more information on the category $\mathcal{C}$. Therefore, there is no apparent reason why completely bounded multipliers should correspond to completely bounded maps 
on the von Neumann algebra $\lambda(\mathbb{C}[\cC)])''$. However, in the tube algebra setting, the situation is more convenient. Indeed, whenever $\varphi: \Irr(\mathcal{C}) \to \mathbb{C}$ is a function on the irreducibles of $\mathcal{C}$ and $\Lambda$ is a full family of objects, then there is a canonical linear map $M_{\varphi}: \cA \Lambda \to \cA \Lambda$ given by
\[
  M_{\varphi}(x) = \varphi(\gamma) x \ \ \ \text{whenever} \ \ \ x \in \Mor(X_{\gamma} \ot Y_{\al}, Y_{\be} \ot X_{\gamma}).
\]
Let us recall here that a multiplier $\varphi: \Irr(\cC) \to \C$ is called completely bounded if 
$\| \varphi \|_{cb} =  \sup_{\al, \beta \in \cC}  \| \theta^{\varphi}_{\al, \beta} \|_{cb}  < \infty$, where $(\theta^{\varphi}_{\al, \beta})_{\al, \beta \in \cC}$ denotes the family of linear maps associated to $\varphi$ as in Definition \ref{mult} and the discussion thereafter.
In terms of the maps $M_{\varphi}$ on the level of the tube algebra, the characterization of completely bounded multipliers is analogous to the group case. This leads to the following proposition which was proven by Vaes and the first-named author \cite[Proposition 5.1]{aranovaes}.
\begin{proposition} \label{ArVaes}
Let $\cC$ be a rigid $C^{\ast}$-tensor category, let $\Lambda$ be a full family of objects of $\cC$, and let $\varphi: \Irr(\cC) \to \C$ be a function. Moreover, let $M_{\varphi}: \cA \Lambda \to \cA \Lambda$ be defined as above. Then $\|M_{\varphi}\|_{\mathrm{cb}} = \| \varphi \|_{\mathrm{cb}}$. If this cb-norm is finite, then $M_{\varphi}$ extends uniquely to a normal completely bounded map on $ \cA \Lambda'' \subset B(L^2(\cA \Lambda))$.
\end{proposition}
\begin{corollary} \label{CorCB}
Let $\varphi$ be a completely bounded multiplier. Then,
the multiplication operator
\[ T_{\varphi} : A(\cC) \to A(\cC), \qquad \theta \mapsto \varphi \theta \quad (\theta \in A(\cC)) \]
is well defined and completely bounded with $\| T_{\varphi} \|_{\mathrm{cb}} \leq \| \varphi \|_{\mathrm{cb}}$.
\end{corollary}
\begin{proof}
The dual map of the multiplication operator $T_{\varphi}$ is given by restricting the map $M_{\varphi}$ to $A(\cC)^*$. By the previous proposition and standard results in operator space theory, the map $T_{\varphi}$ is completely bounded with $ \| T_{\varphi} \| = \| T^*_{\varphi} \| \leq \| \varphi \|_{\mathrm{cb}}$.  
\end{proof}
\begin{corollary} \label{dual}
Let $\cC$ be a rigid $C^{\ast}$-tensor category. Then $M_0A(\cC)$ carries the structure of a dual Banach algebra if we endow it with pointwise addition and multiplication and the cb-norm $\| \cdot \|_{\mathrm{cb}}$.
\end{corollary}
\begin{proof}

Pick a full family of objects $\Lambda$, say $\Lambda = \Irr(\cC)$, and denote the reduced $C^{\ast}$-algebra of $\cA \Lambda$ by $A$ and its enveloping von Neumann algebra by $M = \cA \Lambda''$.  It follows from a well-known result in operator theory due to Effros and Ruan \cite{effrosruan} and, independently, due to Blecher and Paulsen \cite{blecherpaulsen}, that the space of completely bounded maps $CB(A,M)$ is a dual operator space with predual $A \hat{\ot} M_*$. Here, $\hat{\ot}$ denotes the projective tensor product of operator spaces (see \cite[Chapter 4]{pisier} for details). Let us show that the image of the isometric embedding $ M_0A(\cC) \to CB(A,M), \ \varphi \mapsto \widetilde{M}_{\varphi}$ is w$^*$-closed in $CB(A,M)$, where $\widetilde{M}_{\varphi}$ denotes the unique extension of $M_{\varphi}$ to $A$. This will then imply that $ M_0A(\cC)$ is isomorphic as a Banach space to the dual of a quotient of $A \hat{\ot} M_*$ and in particular to a dual Banach algebra. So, let $(\varphi_{i})$ be a net in $M_0A(\cC)$ such that $(\widetilde{M}_{\varphi_{i}})$ converges to a completely bounded map $\Psi \in CB(A,M)$. In particular, this means that 
\[ \omega(\widetilde{M}_{\varphi_{i}}(x)) \to \omega(\Psi(x)) \quad \text{as} \quad i \to \infty \]
for all $x \in A, \ \omega \in M_*$. By choosing $x \in \Mor(X_{\gamma} \ot \1, \1 \ot X_{\gamma})$ and $\omega \in M_*$ such that $\omega(x) \neq 0$ and by applying the definition of $\widetilde{M}_{\varphi_{i}}(x)$, we find that $\varphi_{i}$ converges pointwise to a bounded function $\varphi$. It follows from a short computation that the restriction of $\Psi$ to $\cA \Lambda$ is equal to $M_{\varphi}$. As a consequence, $\varphi$ is completely bounded by the previous proposition with $\Psi = \widetilde{M}_{\varphi}$. Lastly, it is easy to see that pointwise multiplication of completely bounded maps in $M_0A(\cC)$ is separately w$^*$-continuous, so $M_0A(\cC)$ is a dual Banach algebra.
\end{proof}

\section{Leptin's characterization of amenability}
As defined by Popa and Vaes \cite[Definition 5.1]{popavaes}, a rigid $C^{\ast}$-tensor category $\mathcal{C}$ is said to be amenable if there exists a net of finitely supported cp-multipliers $\varphi_i:\Irr(\mathcal{C}) \to \mathbb{C}$ that converges to $1$ pointwise.

In \cite{leptin}, Leptin proved that a locally compact group is amenable if and only if the Fourier algebra of the group admits a bounded approximate unit. We finish this section by proving a version of Leptin's theorem for rigid $C^*$-tensor categories. Before doing so, we note that, using the dimension function $d:\Irr(\cC) \to \C$, one can turn $\Irr(\cC)$ into a discrete hypergroup (see \cite{muruganandam} for the definition of a hypergroup and its Fourier algebra). In the setting of discrete hypergroups, the existence of a bounded approximate unit on the Fourier algebra implies amenability, but the converse implication does not hold (see \cite{alaghmandan}).

We now state our version of Leptin's theorem in the setting of rigid $C^{\ast}$-tensor categories.
\begin{theorem} \label{thm:leptin}
A rigid $C^*$-tensor category $\mathcal{C}$ is amenable if and only if $A(\cC)$ admits a bounded approximate unit, i.e.~a net $(\varphi_{i})$ in $A(\cC)$ such that $\sup_{i} \| \varphi_{i} \|_{A(\cC)} < \infty$ and for all $f \in A(\cC)$,
\[ \| \varphi_{i} f - f\|_{A(\cC)} \to 0 \quad \text{as} \quad i \to \infty. \]
\end{theorem}
In order to prove this theorem, we first prove the following lemma.
\begin{lemma} \label{lem.approx}
The space of finitely supported functions in the unit ball $A(\cC)_1$ is norm dense in $A(\cC)_1$, i.e. $A(\cC)_1 = \overline{c_c(\Irr(\cC)) \cap A(\cC)_1}^{A(\cC)}$.
\end{lemma}
\begin{proof}
Note first that if $\xi, \eta \in c_c(\Irr(\cC))$ are finitely supported functions, the same holds for the matrix coefficient $\varphi_{\xi,\eta}(\al) = d(\al)^{-1} \langle \lambda(\al) \xi, \eta \rangle, \ \al \in \Irr(\cC)$. Since we can approximate any function in $\ell^2(\Irr(\cC))$ by finitely supported ones of smaller norm, every $\varphi \in A(\cC)_1$ can be approximated in norm by functions of the form $\varphi_{\xi,\eta}$ with $\xi, \eta \in c_c(\Irr(\cC))$ and $\| \xi \|, \| \eta \| \leq 1$. More precisely, this follows from the inequality
\[ \| \varphi_{\xi_1,\eta_1} - \varphi_{\xi_2,\eta_2} \| \leq \| \xi_1 - \xi_2 \| \| \eta_1 \| + \| \eta_1 - \eta_2 \| \| \xi_2 \| \]
for all $\xi_i, \eta_i \in A(\cC), \ i=1,2$, which is easily established.
\end{proof}

\begin{proof}[Proof of Theorem \ref{thm:leptin}]
Assume first that $\cC$ is amenable. By \cite[Proposition 5.3]{popavaes}, this means that the trivial representation $\epsilon$ given by $\epsilon(\al) = d(\al), \ \al \in \Irr(\cC)$ extends to a character on $C_r(\cC)$ which we can extend to a (not necessarily normal) state on $C_r(\cC)''$. Since the unit ball of every Banach space is w$^*$-dense in the unit ball of its double dual, there exists a net of normal states $(\omega_{i})$ on $C_r(\cC)''$ such that for all $x \in  C_r(\cC)''$,
\[ \omega_{i}(x) \to \epsilon(x) \quad \text{as} \quad i \to \infty.\]
Let $\varphi_{i} \in A(\cC)$ such that $\omega_{i} = \omega_{\varphi_{i}}$. By the previous lemma, it suffices to show that for all $f \in c_c(\Irr(\cC)) \cap A(\mathcal{C})_1$,
 \[ \| \omega_{\varphi_{i} f - f}\| \to 0 \quad \text{as} \quad i \to \infty. \]
 Let $f \in c_c(\Irr(\cC))$. The operator given by
 \[ T_{f}: \C[\cC] \to \C[\cC], \  T_{f}( \al) = f(\al) \al \quad \al \in \Irr(\cC) \]
 extends to a (completely) bounded finite rank operator on $C_r(\cC)''$ (see Proposition \ref{ArVaes}) with norm $\| T_f \| = K$ for some $K > 0$. We have
 \[ \| \omega_{\varphi_{i} f - f}\| = \sup_{\| x \| \leq 1} | \omega_{\varphi_{i} -1} (T_f(x)) | \leq \sup_{ y \in \ran T_f, \ \|y \| \leq K } | \omega_{\varphi_{i} -1} (y)|.  \]
 But since $\ran T_f$ is finite-dimensional and $\omega_{\varphi_{i} -1} \to 0$ as $i \to \infty$ in the w$^*$-topology, the result follows.\\
 
 Let us now assume that $A(\cC)$ admits a bounded approximate unit $(\varphi_{i})$ with $\| \varphi_{i} \|_{A(\cC)}  \leq 1$ for all $i$. Let $\al \in \Irr(\cC)$. By putting $f = \delta_{\al} \in A(\cC)$, the characteristic function of $\al$, and $x = d(\al)^{-1} \lambda(\al)$, we obtain
 \[ |\varphi_i(\al) -1| \| \lambda(\al) \| = |\omega_{\varphi_{i}f -f} (x) | \to 0 \quad \text{as } i \to \infty \]
 and hence $(\varphi_{i})$ converges to $1$ pointwise. Now, using Lemma \ref{lem.approx}, we can approximate every $\varphi_{i}$ by a net $(\phi^j_{i})$ in $A(\cC)_1 \cap c_c(\Irr(\cC))$ and since $\varphi_{i}$ is a positive element of $A(\cC)$, the function $\phi^j_{i}$ can also be chosen to be positive for all $i$ and $j$. The net $(\phi_{i}^j)_{(i,j)}$ in $A(\cC)_1 \cap c_c(\Irr(\cC))$ converges to $1$ pointwise, which proves the amenability of the category by \cite[Proposition 5.3]{popavaes}.
\end{proof}

\section{Remarks on property (T)} \label{sec:propertyt}
The material of Section \ref{sec:fourieralgebra} gives rise to some observations on property (T) in the setting of $C^{\ast}$-tensor categories that are motivated by Kazhdan's property (T) in the setting of groups. Kazhdan's property (T) is a rigidity property for locally compact groups that has numerous consequences and applications in mathematics. It was introduced in \cite{kazhdan}, in which it was also shown that countable discrete groups with property (T) are finitely generated. Property (T) can be generalized to other settings, such as quantum groups and rigid $C^{\ast}$-tensor categories, and usually the natural analogue of finite generation is still an important consequence of property (T). In particular, Popa and Vaes showed that this is indeed the case in the setting of rigid $C^{\ast}$-tensor categories \cite[Proposition 5.4]{popavaes}.

The definition of property (T) in the context of rigid $C^{\ast}$-tensor categories by Popa and Vaes and two characterizations of this property obtained in \cite{popavaes}, are given in the following definition.
\begin{definition}
A rigid $C^{\ast}$-tensor category $\cC$ has property (T) if one (and hence all) of the following equivalent conditions is satisfied.
\begin{enumerate}[(i)]
\item Every net $(\varphi_{\lambda})$ of cp-multipliers $\varphi_{\lambda} : \Irr \cC \to \mathbb{C}$ converging to $\varphi_{\eps}$ pointwise converges uniformly, i.e. $\sup_{x \in \Irr \cC} |\varphi_{\lambda}(x) -1| \to 0$.
\item If $(\omega_{\lambda})$ is a net of states on $\uniC$ converging to $\eps$ in the weak*-topology, it must already converge in norm.
\item There exists a unique nonzero projection $p \in \uniC $ such that $\alpha p = d(\alpha) p$ for all $\alpha \in \cC$. Such a projection is the analogue of a Kazhdan projection in the setting of groups.
\end{enumerate}
\end{definition}
Let $\cC$ be a rigid $C^{\ast}$-tensor category, and let $W(\cC) = B(\cC)^* = \uniC^{**}$ be the enveloping von Neumann algebra of the full $C^{\ast}$-algebra of $\mathcal{C}$. Since the multiplier $\varphi_{\ep}: \Irr(\cC) \to \mathbb{C}$ given by $\varphi_{\ep}(\al) = 1$ is completely positive by \cite[Corollary 4.4]{popavaes}, the counit $\ep: \C[\cC] \to \C$ extends to a normal $*$-homomorphism on $W(\cC)$.

It is known that for every locally compact group $G$, the Fourier--Stieltjes algebra $B(G)$ has a unique invariant mean. This goes back to \cite[Chapitre III]{godement}. This result was generalized to the setting of locally compact quantum groups in \cite{dawsskalskiviselter}. The next proposition asserts the existence of an invariant mean on the Fourier-Stieltjes algebra of $\cC$, but we formulate it in terms of the existence of a central projection on $W(\mathcal{C})$. 
\begin{proposition}
Let $A$ be a unital $C^*$-algebra and let $\chi: A \to \C$ be a character on $A$.
There exists a unique projection $p$ in the von Neumann algebra $A^{**}$ such that 
\[ x p = p x = \chi(x) p \quad \text{ for all $x \in A^{**}$.} \]
In particular, setting $A = C_u(\cC), \ \chi= \ep$, we find a unique projection $p \in B(\cC)^* = W(\cC)$ such that $\ep(p) = 1$ and $\langle \omega, \al p \rangle = \langle \omega, p \al \rangle = d(\al) \langle \omega, p \rangle$ for all $\al \in \Irr(\cC)$ and $\omega \in B(\cC)$. 
\end{proposition}
\begin{proof}
Uniqueness of $p$ is immediate. To prove the existence, note that, since $\chi$ is a normal $\ast$-homomorphism, its kernel $\ker(\chi)$ is weakly closed and therefore a von Neumann algebra itself. Denote its unit by $e_{\chi}$. Then the central cover $p = 1 - e_{\chi}$ of $\chi$ is a projection in $A^{**}$ satisfying $qp = pq = p$ for all $q$ with $\chi(q) = 1$. On the other hand, if $q$ is a projection in $\ker(\chi)$, we have $pq =0$. Since every von Neumann algebra is the norm closure of the span of its projections, and $\chi$ is in particular norm continuous, the result follows.
\end{proof}
\begin{remark}
  In the group case, it was shown in \cite[Lemma 1]{akemannwalter} (see also \cite[Lemma 3.1]{valette} and \cite[Proposition 4.1]{haagerupknudbydelaat}) that a locally compact group $G$ has Kazhdan's property (T) if and only if the unique invariant mean on $B(G)$ is weak$^{\ast}$-continuous, i.e.~the mean is an element of $C^{\ast}(G)$ rather than just $C^{\ast}(G)^{\ast\ast}$. In fact, under the natural map from $B(G)^{\ast}$ to $C^{\ast}(G)^{\ast\ast}$, the mean is mapped to the Kazhdan projection, which is by weak$^{\ast}$-continuity actually an element of $C^{\ast}(G)$.

  By characterization (iii) above, we see that the same thing happens for $C^{\ast}$-tensor categories: a rigid $C^{\ast}$-tensor category $\mathcal{C}$ has property (T) if and only if the mean on $B(\mathcal{C})$ is weak$^{\ast}$-continuous.
\end{remark}
\begin{remark}
  In the group case, the unique invariant mean on $B(G)$ is the restriction to $B(G)$ of the unique invariant mean on the space $\mathrm{WAP}(G)$ of weakly almost periodic functions on $G$, which is well-known to have a unique invariant mean. Indeed, note that $B(G) \subset \mathrm{WAP}(G)$. Hence, the only thing one needs to show is that this restriction is the unique invariant mean on $B(G)$. In a similar fashion, it is shown (see \cite[Theorem A]{haagerupknudbydelaat}) that the space $M_0A(G)$ of completely bounded Fourier multipliers on $G$ admits a unique invariant mean, using that $B(G) \subset M_0A(G) \subset \mathrm{WAP}(G)$. It is not known whether the space $M_0A(\mathcal{C})$ of a rigid $C^{\ast}$-tensor category admits a unique invariant mean, in particular because it is not known what the natural analogue of $\mathrm{WAP}(G)$ for rigid $C^{\ast}$-tensor categories should be. For locally compact quantum groups, WAP algebras were studied more thoroughly in \cite{dasdaws}. However, to the authors' knowledge, the existence of an invariant mean on the WAP algebra of a locally compact quantum group $\G$ is only known in the case where $\G$ is amenable \cite{runde}.  

  The unique invariant mean on $M_0A(G)$ leads in \cite{haagerupknudbydelaat} to the notion of property (T$^{\ast}$), defined in terms of the mean on $M_0A(G)$ being weak$^{\ast}$-continuous, which obstructs the Approximation Property of Haagerup and Kraus (see \cite{haagerupkraus}). The first examples of groups without the latter property were provided only recently (see \cite{lafforguedelasalle}, \cite{haagerupdelaat1}, \cite{haagerupdelaat2}, \cite{haagerupknudbydelaat} and \cite{liao}). It is still an open problem to find an example of a quantum group without the analogue of the Approximation Property.
\end{remark}
We will now compare property (T) for rigid $C^{\ast}$-tensor categories with other versions of property (T). In the case of discrete quantum groups, we have the following definition of property (T) (see \cite{kyed}), which is equivalent to the one introduced by Fima in \cite{fima}.
\begin{definition} \label{yuki1}
Let $\G$ be a compact quantum group. The discrete dual $\Ghat$ has property (T) if one (and hence all) of the following equivalent conditions is satisfied.
\begin{enumerate}[(i)]
\item If a net of states $(\omega_\lambda)$ in $\uniG^*$ converges to $\eps$ pointwise, then it converges in norm.
\item There exists a projection $p \in \uniG$ such that $x p = \varepsilon(x) p$ for all $x \in \uniG $.
\end{enumerate}
\end{definition}
It was shown in \cite{dawsskalskiviselter} that, also in the more general framework of locally compact quantum groups, condition (i) of the previous definition is equivalent to the conventional notion of property (T) in terms of (almost) invariant vectors. The first part of the following theorem is \cite[Proposition 6.3]{popavaes}, and the second part was proven in \cite{arano1}. 
\begin{theorem} \label{propTtheorem}
Let $\G$ be a compact quantum group. The following conditions are equivalent:
\begin{enumerate}[(i)]
\item the category $\Rep \G$ has property (T) for rigid $C^{\ast}$-tensor categories,
\item the discrete dual $\Ghat$ has central property (T), i.e.~if a net $(\omega_\lambda)$ of central states on $C_u(\mathbb{G})^{\ast}$ converges in the weak$^{\ast}$-topology, then it converges in norm.
\end{enumerate}
Moreover, if we assume the Haar state on $\G$ to be tracial, this is equivalent to the discrete dual $\Ghat$ having (non-central) property (T).
\end{theorem}
We will now complete the picture by involving property (T) for von Neumann algebras. We use the following two characterizations of this property (see \cite[Chapter 12]{brownozawa} for the equivalence).
\begin{definition}
A finite von Neumann algebra $(M,\tau)$ has property (T) if one (and hence all) of the following equivalent conditions is satisfied.
\begin{enumerate}[(i)]
\item If $(\Phi_{\lambda}: M \to M)$ is a net of unital completely positive $\tau$-preserving maps converging to the identity pointwise on $L^2(M)$, i.e.~$ \| \Phi_{\lambda}(x) - x \|_2 \to 0, \ \lambda \to \infty $ for all $x \in M$, then it already converges in norm, i.e.
\[
  \sup_{x \in M_1} \| \Phi_{\lambda}(x) - x \|_2 \to 0 \quad \textrm{as} \quad \lambda \to \infty.
\]
\item For any $M$-bimodule $\mathcal H$ and any net $(\xi_{\lambda})$ of unit vectors satisfying
\[
  \langle x \xi_{\lambda} y, \xi_{\lambda} \rangle_{\mathcal H}  \to  \tau(xy) \quad \text{as} \quad \lambda \to \infty
\]
for all $x,y \in M$ and $\tau(x) = \langle x \xi_{\lambda}, \xi_{\lambda} \rangle = \langle  \xi_{\lambda} x, \xi_{\lambda} \rangle$ for all $\lambda$, there exists a net of $M$-central vectors $(\mu_{\lambda})$ with
\[
  \| \xi_{\lambda} - \mu_{\lambda} \| \to 0 \quad \textrm{as} \quad \lambda \to \infty.
\] 
\end{enumerate}
\end{definition}
The following theorem is a generalization of \cite[Theorem 3.1]{fima}. However, as Stefaan Vaes pointed out to us, the proof in \cite{fima} contains a mistake. Indeed, at a critical point in the proof of \cite[Theorem 3.1]{fima}, it is stated that for two irreducible objects $x,y \in \Irr(\G)$, one has $x \subset x \ot y$ if and only if $y = \1$. This is false, whenever $\hat{\G}$ is not a group. We thank Stefaan Vaes for providing us with a new proof, which we include here with his kind permission.
\begin{theorem} \label{thm:fima}
Let $\G$ be a compact quantum group with a tracial Haar state. Then $\Ghat$ has (central) property (T) if and only if $L^{\infty}(\mathbb{G})$ has property (T).
\end{theorem}
\begin{proof}
Suppose that $\Ghat$ has property (T), and let $\mathcal H$ be a $\vNaG$-bimodule and $(\xi_{\lambda})$ a net of unit vectors in $\mathcal H$ such that $\langle x \xi_{\lambda} y , \xi_{\lambda} \rangle \to h(xy) \ \forall x,y \in \vNaG$ and $ h(x) = \langle x \xi_{\lambda}, \xi_{\lambda} \rangle = \langle  \xi_{\lambda} x, \xi_{\lambda} \rangle$ for all $\lambda$ and $x\in M$. We have to find a net $(\mu_{\lambda})$ of $\vNaG$-central vectors such that $ \| \xi_{\lambda} - \mu_{\lambda} \| \to 0$. For every $\pi \in \Irr(\G)$, choose a unitary matrix $u^{\pi} = (u_{ij}^{\pi})$ representing $\pi$. Since the Haar state is tracial, we can assume that $u^{\bar{\pi}} = \overline{u^{\pi}}$.
Define the linear map
\[ \Theta: \cO(\G) \to B(\cH) \ ; \ \Theta(u_{ij}^{\pi}) \xi = \sum_{k=1}^{d(\pi)} u_{ik}^{\pi} \xi ( u_{jk}^{\pi})^{*}, \ \ \ (\pi \in \Irr(\G)) \]
and, denoting the the coinverse of $\G$ by $S$, observe that $\Theta = \vartheta \circ \Delta$ where $\vartheta: \cO(\G) \ot \cO(\G) \to B(\cH)$ is the $*$-homomorphism defined by $\vartheta(a \ot b) \xi = a \xi S(b), \ \xi \in \cH$.
Hence $\Theta$ is a $*$-homomorphism as well and therefore extends to $\uniG$. Moreover, the conditions on $(\xi_{\lambda})$ imply
\[ \| \Theta(x)\xi_{\lambda} - \eps(x) \xi_{\lambda} \| \to 0 \ \ \ \forall x \in \uniG. \]
Indeed, it suffices to show this for $x$ being a coefficient of a irreducible corepresentation $\pi \in \Irr(\G)$ and in that case one computes
\begin{align*}
\| \Theta(u_{ij}^{\pi})\xi_{\lambda} - \delta_{ij} \xi_{\lambda} \|^2 \xrightarrow{\lambda} \sum_{k,l=1}^{d(\pi)} h((u_{il}^{\pi})^*u_{ik}^{\pi}(u_{jk}^{\pi})^*u_{jl}^{\pi}) - 2 \sum_{k=1}^{d(\pi)} h(u_{ik}^{\pi}(u_{jk}^{\pi})^*) + \delta_{ij} = 0.
\end{align*}
Since $\Ghat$ has property (T), by Definition \ref{yuki1}, we can find a projection $q \in \uniG$ such that $xq = \varepsilon(x) q$ for all $x \in \uniG$ and in particular we have $\varepsilon(q) =1$. Defining $\mu_{\lambda} = \Theta(q) \xi_{\lambda}$, it follows that $ \| \xi_{\lambda} - \mu_{\lambda} \| \to 0$. It only remains to prove that the vector $\mu_{\lambda}$ is $\vNaG$-central for every $\lambda$. To see this, observe first that for $\pi \in \Irr(\G)$, we have
\[ \sum_{k=1}^{d(\pi)} u_{ik}^{\pi} \mu_{\lambda}(u_{jk}^{\pi})^* = \Theta(u_{ij}^{\pi}) \mu_{\lambda} = \Theta(u_{ij}^{\pi}q) \xi_{\lambda} = \delta_{ij} \mu_{\lambda}. \]
Therefore, the computation
\begin{align*}
\mu_{\lambda}u_{il}^{\pi} = \sum_{j=1}^{d(\pi)} \delta_{ij} \mu_{\lambda}u_{jl}^{\pi} 
= \sum_{j,k=1}^{d(\pi)} u_{ik}^{\pi} \mu_{\lambda}(u_{jk}^{\pi})^* u_{jl}^{\pi}
= u_{il}^{\pi} \mu_{\lambda}
\end{align*}
for $\pi \in \Irr(\G), \ i,l= 1, \dots, d(\pi),$ concludes the argument.\\

Let us now assume that $\vNaG$ has property (T). We prove that $\Rep \G$ has property (T), which is equivalent to central property (T) by Theorem \ref{propTtheorem}. Let $(\varphi_{\lambda})_{\lambda}$ be a net of cp-multipliers converging to $\varepsilon$ pointwise. Without loss of generality, we can assume that $\varphi_{\lambda}(1) = 1$ for all $\lambda$. By Proposition 6.1 in \cite{popavaes}, we obtain a net of $h$-preserving unital completely positive maps $\Psi_{\lambda}: \vNaG \to \vNaG$ such that $\Psi_{\lambda}(u_{ij}^{\pi}) = \varphi_{\lambda}(\pi) u_{ij}^{\pi} $ for all $\pi \in \Irr(\G), \ i,j = 1, \dots, \dim \pi$. The pointwise convergence of the net $(\varphi_{\lambda})_{\lambda}$ then implies that the unital completely positive maps $\Psi_{\lambda}: \vNaG \to \vNaG  $ converge pointwise to the identity, i.e.
\[ \| \Psi_{\lambda}(x) - x \|_2 \to 0, \ \ \ \forall x \in \vNaG \ \ \ \text{as} \ \ \ \lambda \to \infty. \]
It follows from the assumption that $\vNaG$ has property (T) that 
\[ \sup_{x \in \vNaG_1}\| \Psi_{\lambda}(x) - x \|_2 \to 0 \ \ \ \text{as} \ \ \ \lambda \to \infty. \]
Now, for all $\pi \in \Irr(\G)$ and all $\lambda$, consider the unital completely positive map 
\[ \id_{\pi} \ot \Psi_{\lambda}: B(H_{\pi}) \ot \vNaG \to  B(H_{\pi}) \ot \vNaG \]
and note that $(\id_{\pi} \ot \Psi_{\lambda})(u^{\pi}) = \varphi_{\lambda}(\pi) u^{\pi} $. Hence,
\[ \sup_{\pi \in \Irr (\G)} |\varphi_{\lambda}(\pi) -1| = \sup_{\pi \in \Irr (\G)} \| (\varphi_{\lambda}(\pi) -1) u^{\pi} \|_2 =  \sup_{\pi \in \Irr (\G)} \| (\id_{\pi} \ot \Psi_{\lambda})(u^{\pi}) - u^{\pi} \|_2 \to 0, \]
which establishes property (T) in the categorial sense.
\end{proof}


\begin{thebibliography}{99}
  \bibitem{akemannwalter}
  C.A.~Akemann and M.E.~Walter, \emph{Unbounded negative definite functions}, Canad.~J.~Math.~\textbf{33} (1981), 862--871.
  
  \bibitem{alaghmandan}
  M.~Alaghmandan, \emph{Amenability notions of hypergroups and some applications to locally compact groups}, Math.~Nachr.~\textbf{290} (2017), 2088--2099.

  \bibitem{arano1}
  Y.~Arano, \emph{Unitary spherical representations of Drinfeld doubles}, J.~Reine Angew.~Math., published online (2016).
  
  \bibitem{aranodelaatwahl}
  Y.~Arano, T.~de Laat and J.~Wahl, \emph{Howe-Moore type theorems for quantum groups and rigid $C^{\ast}$-tensor categories}, Compos.~Math.~\textbf{154} (2018), 328--341.

  \bibitem{aranovaes}
  Y.~Arano and S.~Vaes, \emph{$C^{\ast}$-tensor categories and subfactors for totally disconnected groups}, Operator Algebras and Applications, The Abel Symposium 2015, Springer, 2016, pp.~1--43.

  \bibitem{blecherpaulsen}
  D.~Blecher and V.~Paulsen, \emph{Tensor products of operator spaces}, J.~Funct.~Anal.~\textbf{99} (1991), 262--292.
  
  \bibitem{brownozawa}
  N.P.~Brown and N.~Ozawa, \emph{$C^{\ast}$-Algebras and Finite-Dimensional Approximation Properties}, Amer.~Math.~Soc., Providence, 2008.
  
  \bibitem{dasdaws}
  B. Das and M. Daws, \emph{Quantum Eberlein compactifications and invariant means}, Indiana Univ. Math. J.  \textbf{65}  (2016), 307--352.  
  
  \bibitem{daws}
  M.~Daws, \emph{Multipliers, self-induced and dual Banach algebras}, Dissertationes Math. (Rozprawy Mat.)~\textbf{470} (2010), 62 pp.
  
  \bibitem{dawsskalskiviselter}
  M.~Daws, A.~Skalski and A.~Viselter, \emph{Around property (T) for quantum groups}, Comm.~Math.~Phys.~\textbf{353} (2017), 69--118.
  
  \bibitem{decommerfreslonyamashita}
  K.~De Commer, A.~Freslon and M.~Yamashita, \emph{CCAP for universal discrete quantum groups}, Comm.~Math.~Phys.~\textbf{331} (2014), 677--701.
  
  \bibitem{effrosruan}
  E.~Effros and Z.J.~Ruan, \emph{A new approach to operator spaces}, Canadian Math.~Bull.~\textbf{34} (1991), 329--337.
  
  \bibitem{eymard}
  P.~Eymard, \emph{L'alg\`ebre de Fourier d'un groupe localement compact}, Bull.~Soc.~Math.~France \textbf{92} (1964), 181--236.
  
  \bibitem{fima}
  P.~Fima, \emph{Property T for discrete quantum groups}, Internat.~J.~Math.~\textbf{21} (2010), 47--65.
  
  \bibitem{ghoshjones}
  S.K.~Ghosh and C.~Jones, \emph{Annular representation theory for rigid $C^{\ast}$-tensor categories}, J.~Funct.~Anal.~\textbf{270} (2016), 1537--1584.
  
  \bibitem{godement}
  R.~Godement, \emph{Les fonctions de type positif et la th\'eorie des groupes}, Trans.~Amer.~Math.~Soc.~\textbf{63} (1948), 1--84.
  
  \bibitem{haagerupknudbydelaat}
  U.~Haagerup, S.~Knudby and T.~de Laat \emph{A complete characterization of connected Lie groups with the Approximation Property}, Ann.~Sci.~\'Ec.~Norm.~Sup\'er.~(4) \textbf{49} (2016), 927--946.
  
  \bibitem{haagerupkraus}
  U.~Haagerup and J.~Kraus, \emph{Approximation properties for group {$\operatorname{C}^*$-algebras} and group von {Neumann} algebras}, Trans.~Amer.~Math.~Soc. \textbf{344} (1994), 667--699.

  \bibitem{haagerupdelaat1}
  U.~Haagerup and T.~de Laat, \emph{Simple Lie groups without the Approximation Property}, Duke Math.~J.~\textbf{162} (2013), 925--964.

  \bibitem{haagerupdelaat2}
  U.~Haagerup and T.~de Laat, \emph{Simple Lie groups without the Approximation Property II}, Trans.~Amer.~Math.~Soc.~\textbf{368} (2016), 3777--3809.
  
  \bibitem{cjones}
  C.~Jones, \emph{Quantum $G_2$ categories have property (T)}, Internat.~J.~Math.~\textbf{27} (2016), 1650015.

  \bibitem{kazhdan}
  D.A.~Kazhdan, \emph{Connection of the dual space of a group with the structure of its closed subgroups}, Funk.~Anal.~Appl.~\textbf{1} (1967), 63--65.

  \bibitem{kyed}
  D.~Kyed, \emph{A cohomological description of property (T) for quantum groups}, J.~Funct.~Anal.~\textbf{261} (2011), 1469--1493.
  
  \bibitem{lafforguedelasalle}
  V.~Lafforgue and M.~de la Salle, \emph{Noncommutative {$L^p$}-spaces without the completely bounded approximation property}, Duke.~Math.~J.~\textbf{160} (2011), 71--116.
  
  \bibitem{leptin}
  H.~Leptin, \emph{Sur l'alg\`ebre de Fourier d'un groupe localement compact}, C.~R.~Acad.~Sci.~Paris S\'er.~A \textbf{266} (1968), 1180--1182.

  \bibitem{liao}
  B.~Liao, \emph{Approximation properties for $p$-adic symplectic groups and lattices}, preprint (2015), arXiv:1509.04814.

  \bibitem{muruganandam}
  V.~Muruganandam, \emph{Fourier algebra of a hypergroup. I}, J.~Aus.~Math.~Soc. \textbf{82(1)} (2007), 59--83.

  \bibitem{neshveyevtuset}
  S.~Neshveyev and L.~Tuset, \emph{Compact Quantum Groups and their Representation Categories}, Soci\'et\'e Math\'ematique de France, Paris, 2013.
  
  \bibitem{neshveyevyamashita}
  S.~Neshveyev and M.~Yamashita, \emph{Drinfeld center and representation theory for monoidal categories}, Comm.~Math.~Phys.~\textbf{345} (2016), 385--434.
  
  \bibitem{ocneanu}
  A.~Ocneanu, \emph{Chirality for operator algebras}, Subfactors (Kyuzeso, 1993), World Sci.~Publ., River Edge, 1994, pp.~39--63.
  
  \bibitem{pisier}
  G.~Pisier, \emph{Introduction to Operator Spaces}, Cambridge University Press, Cambridge, 2003.
  
  \bibitem{popavaes}
  S.~Popa and S.~Vaes, \emph{Representation theory for subfactors, $\lambda$-lattices and $C^{\ast}$-tensor categories}, Comm.~Math.~Phys.~\textbf{340} (2015), 1239--1280.
  
  \bibitem{popashlyakhtenkovaes}
  S.~Popa, D.~Shlyakhtenko and S.~Vaes, \emph{Cohomology and $L^2$-Betti numbers for subfactors and quasi-regular inclusions}, to appear in Int.~Math.~Res.~Not.
  
  \bibitem{runde}
  V.~Runde, \emph{Uniform continuity over locally compact quantum groups}, J.~London Math.~Soc.~\textbf{80} (2009), 55--71.

  \bibitem{takesaki}
  M.~Takesaki, \emph{Theory of Operator Algebras II}, Springer-Verlag, Berlin, 2003.
  
  \bibitem{tarragowahl}
  P.~Tarrago and J.~Wahl, \emph{Free wreath product quantum groups and standard invariants of subfactors}, preprint (2016), arXiv:1609.01931.
  
  \bibitem{valette}
  A.~Valette, \emph{Minimal projections, integrable representations and property (T)}, Arch.~Math.~(Basel) \textbf{43} (1984), 397--406.
\end{thebibliography}
\end{document}